\documentclass[11pt]{article}
\usepackage{amssymb,amsmath,amsfonts,amsthm,authblk,hyperref,here,enumerate,xfrac,mathtools,graphicx,pstool,xcolor,pgf,tikz,array}
\usetikzlibrary{arrows,arrows.meta}
\usepackage[main=english,ngerman]{babel}

\numberwithin{equation}{section}

\newcommand{\R}{\mathbb R}
\newcommand{\N}{\mathbb N}

\definecolor{darkgreen}{rgb}{0,.6,0}

\definecolor{MyDarkBlue}{rgb}{0,0.29,0.7}
\hypersetup{
    colorlinks=true,       % false: boxed links; true: colored links
    linkcolor=MyDarkBlue,          % color of internal links (change box color with linkbordercolor)
    linkbordercolor=MyDarkBlue,
    citecolor=MyDarkBlue,
    citebordercolor=MyDarkBlue,        % color of links to bibliography
    filecolor=MyDarkBlue,      % color of file links
    urlcolor=MyDarkBlue           % color of external links
}

\theoremstyle{plain}

\newtheorem{theorem}{Theorem}[section]
\newtheorem{coro}[theorem]{Corollary}
\newtheorem{lemma}[theorem]{Lemma}
\newtheorem{prop}[theorem]{Proposition}

\theoremstyle{definition}

\newtheorem{definition}[theorem]{Definition}
\newtheorem{example}[theorem]{Example}

\tikzset{
  on each segment/.style={
    decorate,
    decoration={
      show path construction,
      moveto code={},
      lineto code={
        \path [#1]
        (\tikzinputsegmentfirst) -- (\tikzinputsegmentlast);
      },
      curveto code={
        \path [#1] (\tikzinputsegmentfirst)
        .. controls
        (\tikzinputsegmentsupporta) and (\tikzinputsegmentsupportb)
        ..
        (\tikzinputsegmentlast);
      },
      closepath code={
        \path [#1]
        (\tikzinputsegmentfirst) -- (\tikzinputsegmentlast);
      },
    },
  },
  mid arrow/.style={postaction={decorate,decoration={
        markings,
        mark=at position .5 with {\arrow[#1]{>}}
      }}},
}
  
\newcolumntype{L}{>{$\displaystyle} l <{$}}
\begin{document}

\title{An integral that counts the zeros of a function}
\author[1]{Norbert Hungerb\"uhler}
\author[2]{Micha Wasem}
\affil[1]{Department of Mathematics, ETH Z\"urich, R\"amistrasse 101, 8092 Z\"urich, Switzerland}
\affil[2]{HTA Freiburg, HES-SO University of Applied Sciences and Arts Western Switzerland, P\'erolles 80, 1700 Freiburg, Switzerland}
\date{\today}
\maketitle
\begin{abstract}
\noindent Given a real function $f$ on an interval $[a,b]$ satisfying mild regularity
conditions, we determine the number of zeros of $f$
by evaluating a certain integral. The integrand depends on $f, f'$ and $f''$.
In particular, by approximating the integral with the trapezoidal rule on
a fine enough grid,
we can compute the number of zeros 
of  $f$  by evaluating finitely many values of $f,f'$ and  $f''$. A variant
of the integral
even allows to determine the number of the zeros broken down by their multiplicity.
\end{abstract}
\hspace{5ex}{\small{\it Key words\/}: number of zeros on an interval, multiplicity of zeros}

\hspace{5ex}{\small{\it 2010 Mathematics Subject 
Classification\/}: {\bf 30C15}
%\tableofcontents
\normalsize
\section{Introduction}
Counting the zeros of a given function $f$ in a certain region
belongs to the basic tasks in analysis. If $f:\mathbb C\to\mathbb C$
is holomorphic, the Argument Principle and Rouch\'e's Theorem
are tools which allow to find the number of zeros of $f$,
counted with multiplicity, in a bounded domain of $\mathbb C$ with
sufficiently regular boundary (see, e.g.~\cite{zeros-analytic} for an overview 
of methods used for analytic functions). Descartes' Sign Rule
is a method of determining the maximum number of 
positive and negative real roots (counted with multiplicity) of a polynomial. 
The Fourier-Budan Theorem yields the maximum number of roots 
(counted with multiplicity) of a polynomial in an interval.
Sturm's Theorem, a refinement of Descartes' Sign Rule and the Fourier-Budan Theorem,
allows to count the exact number of distinct roots of a polynomial
on a real interval (see, e.g.,~\cite{obreschkoff}, \cite{henrici}, \cite{sturmfels}). 
The mentioned methods are restricted to holomorphic
functions and polynomials, respectively. On the other end of the regularity spectrum, 
for a merely continuous function $f$,
the Theorem of Bolzano yields the information that
at least one zero exists on an interval $[a,b]$ if
$f$ has opposite signs at its endpoints, though, it does not count the zeros.
Here, we want to construct a method which 
gives the number of zeros of a real function under only mild 
regularity assumptions. More precisely, we
want to express the number of zeros of a function
$f%\in C^2([a,b])
$ by a certain integral (and boundary terms). The integrand depends on $f, f'$ and $f''$.
If $f%\in C^4([a,b])
$ is sufficiently regular, the integral (and hence the number of
zeros of $f$) can be expressed by evaluating the integrand
on a sufficiently fine partition of $[a,b]$. Modifications of the integral
even allow to determine the number of the zeros broken down by their multiplicity.

To explain the basic idea, we consider the following elementary 
connection between the number of zeros of a periodic function
and the winding number of the related kinematic curve in the state space
with respect to the origin:
\begin{lemma}\label{lemmaidee}
Let $f : \mathbb R\to \mathbb R$ be a $2\pi$-periodic $C^2$ function
with only simple zeros, i.e.~points $x$ with $f(x)=0\neq f'(x)$.
Then, the number $n$ of zeros of $f$ in $[0,2\pi)$ equals 
twice the %negative 
winding number of the curve 
% $\gamma: [0,2\pi)\to \mathbb R^2, x\mapsto (f(x),f'(x))$ 
$\gamma: [0,2\pi)\to \mathbb R^2, x\mapsto (f'(x),f(x))$ 
with respect to the origin. Hence
$$
n=\frac1{\pi}\int_0^{2\pi}\frac{f'(x)^2-f(x)f''(x)}{f(x)^2+f'(x)^2} \,\mathrm dx.
$$
\end{lemma}
Figure~\ref{fig:winding} illustrates a heuristic proof without words: Each colored arc
between two zeros of $f$ adds %$-\frac12$ 
$\frac12$ 
to the winding number of $\gamma$.  In the sequel,
we will rigorously prove much more general versions and variants
of this result. We will develop integrals that count the number of 
zeros with and without multiplicity, and we will even be able to determine
the number of zeros of a given multiplicity. As a byproduct, a coherent definition
of a fractional multiplicity of zeros will be possible. %Moreover, 
%we can generalize the winding number of a curve to points
%sitting on the curve. This allows to 
%extend the residue theorem to curves passing through singularities.
To start with, it is necessary to analyze the nature of zeros of a function. 
\begin{figure}[h]\label{fig:winding}
\begin{center}
\begin{tikzpicture}[scale=.58,domain=0:6.28319]
\draw[->] (-.2,0) -- (7,0)  node[anchor=west] {\small$x$};
\draw (6.28319,0) node[anchor = south] {\small $2\pi$};
%\draw[->] (0,-3.3) -- (0,3)  node[anchor=east] {\small$f(x)$};
\draw[->] (0,-3.6) -- (0,2.8)  node[anchor=east] {\small$f(x)$};
\draw (6.28319,-.1)--(6.28319,.1);
\draw [dotted](6.28319,0 )--(6.28319,-1.8);
\draw[thick,smooth,color=blue,domain=0:0.468218] plot (\x,{.2*(8*sin(\x r)   -6*cos(3*\x r) + .4*sin(4*\x r) - 3)});
\draw[thick,smooth,color=red,domain=0.468218:1.89583] plot (\x,{.2*(8*sin(\x r)   -6*cos(3*\x r) + .4*sin(4*\x r) - 3)});
\draw[thick,smooth,color=darkgreen,domain=1.89583:2.55344] plot (\x,{.2*(8*sin(\x r)   -6*cos(3*\x r) + .4*sin(4*\x r) - 3)});
\draw[thick,smooth,color=cyan,domain=2.55344:3.37821] plot (\x,{.2*(8*sin(\x r)   -6*cos(3*\x r) + .4*sin(4*\x r) - 3)});
\draw[thick,smooth,color=blue,domain=3.37821:6.28319] plot (\x,{.2*(8*sin(\x r)   -6*cos(3*\x r) + .4*sin(4*\x r) - 3)});
\end{tikzpicture}\qquad
% \begin{tikzpicture}[scale=.58,domain=0:6.28319]
% \draw[->] (-3.9,0) -- (2.6,0)  node[anchor=west] {\small$f(x)$};
% \draw[->] (0,-3.3) -- (0,3)  node[anchor=east] {\small$f'(x)$};
% \draw (1.85,1.55) node[anchor=center] {\small$\gamma$};
% \draw[thick,smooth,color=blue,domain=0:0.468218] plot ({.2*(8*sin(\x r)-6*cos(3*\x r)+.4*sin(4*\x r)-3)},{.1*(8*cos(\x r)+1.6*cos(4*\x r)+18*sin(3*\x r))});
% \draw[thick,smooth,color=red,domain=0.468218:1.89583] plot ({.2*(8*sin(\x r)-6*cos(3*\x r)+.4*sin(4*\x r)-3)},{.1*(8*cos(\x r)+1.6*cos(4*\x r)+18*sin(3*\x r))});
% \draw[-{>[scale=1.3]},color=red]  (1.70281, 1.12309)--(1.80186, 0.856252);
% \draw[thick,smooth,color=darkgreen,domain=1.89583:2.55344] plot ({.2*(8*sin(\x r)-6*cos(3*\x r)+.4*sin(4*\x r)-3)},{.1*(8*cos(\x r)+1.6*cos(4*\x r)+18*sin(3*\x r))});
% \draw[thick,smooth,color=cyan,domain=2.55344:3.37821] plot ({.2*(8*sin(\x r)-6*cos(3*\x r)+.4*sin(4*\x r)-3)},{.1*(8*cos(\x r)+1.6*cos(4*\x r)+18*sin(3*\x r))});
% \draw[thick,smooth,color=blue,domain=3.37821:6.28319] plot ({.2*(8*sin(\x r)-6*cos(3*\x r)+.4*sin(4*\x r)-3)},{.1*(8*cos(\x r)+1.6*cos(4*\x r)+18*sin(3*\x r))});
% \end{tikzpicture}
\begin{tikzpicture}[scale=.58,domain=0:6.28319]
\draw[->] (-3,0) -- (3,0)  node[anchor=west] {\small$f'(x)$};
\draw[->] (0,-3.6) -- (0,2.8)  node[anchor=east] {\small$f(x)$};
\draw (1.85,1.65) node[anchor=center] {\small$\gamma$};
\draw[thick,smooth,color=blue,domain=0:0.468218] plot ({.1*(8*cos(\x r)+1.6*cos(4*\x r)+18*sin(3*\x r))},{.2*(8*sin(\x r)-6*cos(3*\x r)+.4*sin(4*\x r)-3)});
\draw[thick,smooth,color=red,domain=0.468218:1.89583] plot ({.1*(8*cos(\x r)+1.6*cos(4*\x r)+18*sin(3*\x r))},{.2*(8*sin(\x r)-6*cos(3*\x r)+.4*sin(4*\x r)-3)});
\draw[-{>[scale=1.3]},color=red] ( 1.12309,1.70281)-- ( 0.856252,1.80186);
\draw[thick,smooth,color=darkgreen,domain=1.89583:2.55344] plot ({.1*(8*cos(\x r)+1.6*cos(4*\x r)+18*sin(3*\x r))},{.2*(8*sin(\x r)-6*cos(3*\x r)+.4*sin(4*\x r)-3)});

\draw[thick,smooth,color=cyan,domain=2.55344:3.37821] plot ({.1*(8*cos(\x r)+1.6*cos(4*\x r)+18*sin(3*\x r))},{.2*(8*sin(\x r)-6*cos(3*\x r)+.4*sin(4*\x r)-3)});
\draw[thick,smooth,color=blue,domain=3.37821:6.28319] plot ({.1*(8*cos(\x r)+1.6*cos(4*\x r)+18*sin(3*\x r))},{.2*(8*sin(\x r)-6*cos(3*\x r)+.4*sin(4*\x r)-3)});
\end{tikzpicture}
\caption{Number of zeros of $f$ vs.~winding number of $(f',f)$.}
\end{center}
\end{figure}
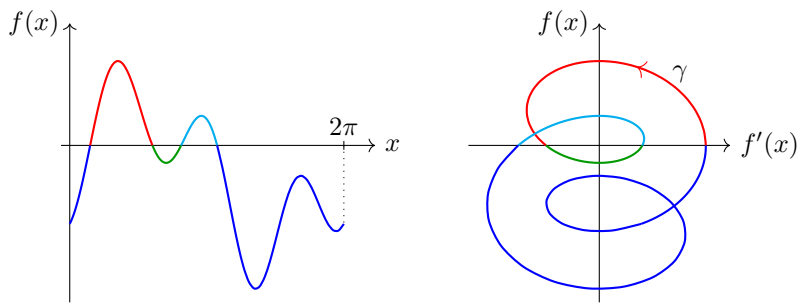
\section{Zeros of Functions}\label{sec-zeros}
A function $f:(a,b)\to\mathbb R$  may, in general, show a quite pathological behavior
in the neighborhood of one of its zeros (see, e.g., Examples~\ref{example-zero}.\ref{ex3} and \ref{ex-counter} below). 
To exclude such exotic cases but still be sufficiently general to
cover most of the relevant cases, we use the following definition.
\begin{definition}\label{def1}
A zero $x_0\in(a,b)$ of a
function $f\in C^0(a,b)\cap C^1((a,b)\setminus\{x_0\})$ will be called {\em admissible} provided
\begin{equation}\label{def-admissible}
\lim_{x\nearrow x_0} \frac{f'(x)}{f(x)}=-\infty \text{ and }\lim_{x\searrow x_0}\frac{f'(x)}{f(x)}=\infty.
\end{equation}
If $f$ extends continuously to $a$ (or $b$) and $f(a)=0$ (or $f(b)=0$), we will say that $f$ has an {\em admissible zero\/} in $a$ (or $b$) if
$$
\lim_{x\searrow a}\frac{f'(x)}{f(x)}=\infty~\left(\text{or }\lim_{x\nearrow b}\frac{f'(x)}{f(x)}=-\infty\right).$$
\end{definition}
{\em Remarks.} 
\begin{enumerate}
\item 
%Apparently, a zero contained in an interval where $f$
%vanishes identically cannot be admissible.
An admissible zero is necessarily an isolated zero. In fact, if the zero $x_0$ 
is an accumulation point of zeros of $f$ then, by Rolle's Theorem, it is also an accumulation
point of zeros of $f'$ and the limits in Definition~\ref{def1} cannot be plus or minus infinity.
\item The condition on the limits given in \eqref{def-admissible} is in fact equivalent to
\begin{equation}\label{def-admissible-equivalent}
\lim_{x\to x_0}\left|\frac{\mathrm d}{\mathrm dx}\ln|f(x)|\right|=\infty.
\end{equation}
Indeed, if \eqref{def-admissible-equivalent} holds true, it follows that $x_0$ is an isolated zero of $f$, hence $f$ does not change its sign on $(x_0,x_0+\varepsilon)$ and on $(x_0-\varepsilon,x_0)$ for $\varepsilon>0$ small enough. Moreover $0<|f(x)|<|f'(x)|$ on a punctured neighborhood of $x_0$. Hence, $f'$ cannot change sign and the claim follows by distinction of cases.
The condition~(\ref{def-admissible-equivalent}) is slightly more compact than~(\ref{def-admissible}), however,
(\ref{def-admissible}) is easier to handle in the calculations below. 

\item A simple zero $x_0\in(a,b)$ of $f\in C^1(a,b)$, i.e.\ $f(x_0)=0$ and $f'(x_0)\ne 0$ is admissible.
It suffices to consider $x_0=0$:
$$
\lim_{x\searrow 0} \frac{f'(x)}{f(x)}=\lim_{x\searrow 0}\frac{f'(0)+o(1)}{f(0)+xf'(0)+o(x)}=\lim_{x\searrow 0}\frac{1}{x}\cdot \frac{f'(0)+o(1)}{f'(0)+o(1)}=\infty.$$
The limit $x\nearrow 0$ is analogous.
\item\label{rem2} If $f(x_0)=f'(x_0)=0$ and $f'$ is monotone on $(x_0,x_0+\epsilon)$ and 
on  $(x_0-\epsilon,x_0)$ for some $\epsilon>0$, then $x_0$ is an admissible zero:
Indeed, for $x_0<x<x_0+\epsilon$ and $f'$ non-decreasing (if $f'$ is non-increasing consider $-f$) on $(x_0,x_0+\epsilon)$,
we have $f(x)=\int_{x_0}^x f'(t)\,\mathrm dt\leqslant(x-x_0)f'(x)$ and thus $\frac{f'(x)}{f(x)}\geqslant \frac{1}{x-x_0}\to \infty$ for
$x\searrow x_0$. The argument for the limit $x\nearrow x_0$ is analogous.
\item If $f\in C^k(a,b)$ and $x_0\in(a,b)$ is a zero of
multiplicity $k>1$, i.e.\ $f^{(\ell)}(x_0)=0$ for all $\ell=0,\ldots, k-1$
and $f^{(k)}(x_0)\ne 0$, then $x_0$ is admissible. 
%This can be shown the same
%way as above replacing the linear approximation by a higher order Taylor expansion.
This follows easily by an iterated application of L'H\^opital's rule.
Hence the zeros of real-analytic functions and a fortiori
zeros of polynomials are admissible.
\item\label{rem5} If $f(x)=|x-x_0|^\alpha g(x)$ for a $C^1$-function $g$ with $g(x_0)\neq 0$ and
$0< \alpha\in\mathbb R$, then $x_0$ is an admissible zero of $f$.
\item Every $f\in C^1([a,b])$ can be extended to $\tilde f\in C^1(I)$, where $I\supset[a,b]$ is an open interval and the limits
\begin{equation}\label{limits}\lim_{x\nearrow a}\frac{f'(x)}{f(x)}\text{ and }\lim_{x\searrow b}\frac{f'(x)}{f(x)}\end{equation}
can be defined via $\tilde f$, provided $f(a),f(b)\ne 0$. 
If $f$ has an admissible zero in $a$ (or $b$), $f$ can be extended antisymmetrically
with respect to $a$ (or $b$) to an extension $\tilde f$ for which $a$ (or $b$) is an admissible zero. We will henceforth use this particular extension when computing limits like in \eqref{limits}.

\end{enumerate}
%We have the following equivalent criterion for the admissibility of a zero.

%%We will now formulate a sharper criterion for %a class of functions in $W^{2,1}(a,b)$ 
%%functions having admissible zeros including also zeros of infinite multiplicity.
%Observe, that the weak derivative $f'$ has a representative which is absolutely
%continuous and which coincides with the classical derivative of
%a $C^1$ representative of $f$. The statement of the Lemma and the arguments in the proof
%refer to these representatives.
%\begin{lemma}\label{isoliert}
%Let $f\in C^1(a,b)$. A zero $x_0\in(a,b)$ of $f$ is admissible if and only if the zeros of $f'$ do not
%% C^1 ist eine schwächere Voraussetzung als W^{2,1}, genügt aber
%% auch für dieses Lemma.
%n isolated 
%% "isolated" braucht hier nicht zusätzlich vorausgesetzt zu werden,
%% da dies aus der Voraussetzung folgt, dass sich die Nullstellen von f' nicht 
%% häufen.
%accumulate at $x_0$.
%\end{lemma}
%\begin{proof} 

\begin{example}\label{example-zero}
\begin{enumerate}
\item The function $f_1\in C^0(\R)\cap C^\infty(\R\setminus\{0\})$, $x\mapsto \sqrt{|x|}$ has an admissible zero in $x=0$
 (see Remark~\ref{rem5} above).
\item \label{expzero}The $C^\infty$-function
$$
f_2(x) := \begin{cases}\exp\left(-\frac{1}{x^2}\right),&x\ne 0 \\ \hfill 0,& x=0,\end{cases}
$$
has an admissible zero of infinite multiplicity at $x=0$ (see Remark~\ref{rem2} above).
%The corresponding limit in~(\ref{def-admissible})
%can be evaluated using Lemma \ref{isoliert} 
%but note that L'H\^opital's rule fails to evaluate it since $f_1$ has a zero of infinite multiplicity at $0$.
%\item
%Another example, where Lemma \ref{isoliert} applies, but L'H\^opital's rule fails
% for a different reason, is given by
%$\lim\limits_{x\nearrow 0}\frac{f_2(x)}{f_2'(x)}$, where
%The function
%$$
%f_2(x) := x^2\left(2-x \sin\left(\frac{1}{x}\right)\right)
%$$
%Here we cannot use L'H\^opital's rule to evaluate $\lim\limits_{x\nearrow 0}\frac{f_2(x)}{f'_2(x)}$
%since $f_2''(x)$ has zeros which accumulate at $0$. Nonetheless, 
%Lemma~\ref{isoliert} reveals that $0$ is an admissible zero of $f_2$, however, the corresponding limits cannot be evaluated using
%L'H\^opital's rule since the limits
%$$
%\lim_{x\nearrow 0}\frac{f_2''(x)}{f_2'(x)}\text{ and }\lim_{x\searrow 0}\frac{f_2''(x)}{f_2'(x)}
%$$
%do not exist (even in an improper sense).

\item\label{ex3} An example of an isolated zero which is {\em not\/} admissible is given by the $C^\infty$-function
$$
f_3(x):=f_2(x)\Bigl(\sin\Bigl(\frac{1}{x^3}\Bigr)+2\Bigr),
$$
which vanishes (together with all derivatives) in $0$ but the 
corresponding limits~(\ref{def-admissible}) do not exist.
\end{enumerate}
\end{example}
%In order to match the regularity assumption of Theorem \ref{main} (\"uberpr\"ufen, was wir genau brauchen),
%es scheint, als müsste man für den Satz main zunächst einmal $C^2$ verlangen 
%we will consider the space $W^{2,1}(a,b)$. 
%
%Note furthermore that the notion of a zero of multiplicity $k$ will extend to functions $f\in W^{k+1,1}(a,b)$ when identified with their representative in $C^k([a,b])$.

\begin{definition}\label{admissiblefunction} A function $f:[a,b]\to \R$ belongs to $\mathcal A^k([a,b])$, $k\in\N$, if the following holds:
\begin{enumerate}
\item $f\in C^0([a,b])$.
\item $f$ has only admissible (and therefore finitely many) zeros $x_1<\ldots<x_n$ and $f|_{(x_i,x_{i+1})}$ ($i=1,\ldots,n-1$), $f|_{(a,x_1)}$ and $f|_{(x_n,b)}$ are of class $C^{k+1}$.
\item There exists a partition $a=y_1<y_2<\ldots <y_m = b$ such that $f|_{(y_i,y_{i+1})}$ is of class $C^{k+2}$ for all $i=1,\ldots,m-1$.
\end{enumerate}
If $f\in\mathcal A^0([a,b])$, $f$ will be called \emph{admissible}. \end{definition}
{\em Remarks.} 
\begin{enumerate}
\item Observe that $\mathcal A^{k+1}([a,b])\subset\mathcal A^{k}([a,b])$ for all $k\in\N$ by construction.
\item Every analytic function is in $\mathcal A^\infty([a,b])$.
\item $f:[-1,1]\to\R$, $x\mapsto \sqrt{|x|}$ is in $\mathcal A^\infty([a,b])$.
\item If $f$ is admissible, then $x\mapsto(f'(x),f(x))$ is not necessarily a continuous curve.
\end{enumerate}

As a building block of the intended results we need the following: For $\sigma\in[-\infty,\infty]$, let
\begin{equation}\label{eq-H}
\mathrm H(x)=\int_{\sigma}^x h(t)\,\mathrm dt,
\end{equation}
where $h:\mathbb R\to\mathbb R$ is any piecewise continuous function % in $L^1(\mathbb R)$
such that the improper integral $\int_{-\infty}^{\infty}h(x)\,\mathrm dx = 1$. Then we have the following theorem (recall \eqref{limits} in order to make sense of the limits that appear).
\begin{theorem}\label{main}Let $f\in\mathcal A^0([a,b])$. The number of zeros $n(f)$ of $f$ in $[a,b]$ is given by
$$
n(f)=\int_a^{b}h\left(\frac{f'(x)}{f(x)}\right)\frac{f'(x)^2-f(x)f''(x)}{f(x)^2}\,\mathrm dx+\lim_{x\searrow b}\mathrm H\left(\frac{f'(x)}{f(x)}\right)-\lim_{x\nearrow a}\mathrm H\left(\frac{f'(x)}{f(x)}\right)
$$
and the number of zeros $\mathring n(f)$ of $f$ in $(a,b)$ by
$$
\mathring n(f) = \int_a^{b}h\left(\frac{f'(x)}{f(x)}\right)\frac{f'(x)^2-f(x)f''(x)}{f(x)^2}\,\mathrm dx+\lim_{x\nearrow b}\mathrm H\left(\frac{f'(x)}{f(x)}\right)-\lim_{x\searrow a}\mathrm H\left(\frac{f'(x)}{f(x)}\right).
$$
\end{theorem}
\begin{proof}
Consider first the case, where $f(a),f(b)\ne 0$. Then the zeros of $f$ are given by $a<x_1<x_2<\ldots < x_{n(f)} <b$. The integrand of
$$
\int_a^{b}h\left(\frac{f'(x)}{f(x)}\right)\frac{f'(x)^2-f(x)f''(x)}{f(x)^2}\,\mathrm dx =: \int_a^{b}\mathrm I(x)\,\mathrm dx$$
is a priori undefined whenever $f$ vanishes or whenever $f''$ is undefined. We decompose the integral and compute the resulting improper integrals using unilateral limits. Since $f$ is admissible, we have
$$
\int_{x_j}^{x_{j+1}}\mathrm I(x)\,\mathrm dx = \lim_{x\searrow x_j}\mathrm H_x-\lim_{x\nearrow x_{j+1}}\mathrm H_x=1
$$
for all $j=1,\ldots, n(f)-1$, where $\mathrm H_x:=\mathrm H(f'(x)/f(x))$. Integrating over a neighborhood of a point $y$ where $f''$ is undefined does not introduce further boundary terms since $
\lim_{x\searrow y} \mathrm H_x - \lim_{x\nearrow y}\mathrm  H_x = 0.$ Hence
\begin{equation}\label{integralcomputation}\begin{aligned}
\int_a^{b}\mathrm I(x)\,\mathrm dx & = \int_a^{x_1}\mathrm I(x)\,\mathrm dx+\sum_{j=1}^{n(f)-1}\int_{x_j}^{x_{j+1}}\mathrm I(x)\,\mathrm dx+\int_{x_{n(f)}}^{b}\mathrm I(x)\,\mathrm dx=\\
 & = \mathrm H_a-\lim_{x\nearrow x_1}\mathrm H_x+(n(f)-1)+\lim_{x\searrow x_{n(f)}}\mathrm H_x-\mathrm H_b
\end{aligned}\end{equation}
and therefore
\begin{equation}\label{interiorzero}
n(f)=\int_a^{b}\mathrm I(x)\,\mathrm dx+\mathrm H_b-\mathrm H_a.
\end{equation}
The computation above suggests that $n(f)>1$ but one can check that formula \eqref{interiorzero} holds true for $n(f)=1$ and $n(f)=0$ as well.

If $f$ has zeros in $a$ and $b$ and therefore $x_1=a, x_{n(f)}=b$, computation \eqref{integralcomputation} gives
\begin{equation}\label{boundaryzerocompute}
n(f)=\int_a^{b}\mathrm I(x)\,\mathrm dx+1.\end{equation}
According to \eqref{limits}, $\lim\limits_{x\searrow b}\mathrm H_x-\lim\limits_{x\nearrow a}\mathrm H_x=1$ and
\eqref{boundaryzerocompute} becomes
\begin{equation}\label{boundaryzero}
n(f)=\int_a^{b}\mathrm I(x)\,\mathrm dx+\lim_{x\searrow b}\mathrm H_x-\lim_{x\nearrow a}\mathrm H_x\end{equation}
and hence \eqref{boundaryzero} counts the zeros of $f$ in $[a,b]$ since it reduces to \eqref{interiorzero} if $f(a),f(b)\ne 0$ and one can check that the remaining cases $f(a)=0\ne f(b)$ and $f(a)\ne 0=f(b)$ are also covered. Let now
$$
\mathring n(f) = \int_a^{b}\mathrm I(x)\,\mathrm dx+\lim_{x\nearrow b}\mathrm H_x-\lim_{x\searrow a}\mathrm H_x.$$
Since
$$\begin{aligned}
n(f) - \mathring n(f) & = \lim_{x\searrow b}\mathrm H_x-\lim_{x\nearrow a}\mathrm H_x-\left(\lim_{x\nearrow b}\mathrm H_x-\lim_{x\searrow a}\mathrm H_x\right)=\\&=\begin{cases}\hfill 0,& \text{if }f(a),f(b)\ne 0 \\
\hfill 1,&\text{if either }f(a)=0\text{ or }f(b)=0\\
2,&\text{if }f(a)=f(b)=0\end{cases}
\end{aligned}$$
we conclude that $\mathring n(f)$ counts the zeros of $f$ in $(a,b)$.
\end{proof}

\emph{Remarks.}
\begin{enumerate}
\item Putting $g(x) := f'(x)/f(x)$, the integrand in Theorem \ref{main} reads $
-(h\circ g)(x) g'(x).$
With respect to the signed Borel-Lebesgue-Stieltjes-Measure $\mathrm dg(x):=g'(x)\,\mathrm dx$ (see \cite{dshalalow}), the integral can be written more compactly as
$$
-\int_a^b h(g)\,\mathrm dg.
$$
\item
If $h(x):=1/(\pi(1+x^2))$, i.e.\ $h$ equals the \emph{Cauchy Density} and $f$ is an admissible $2\pi$-periodic function, then the number $n$ of zeros of $f$ in $[0,2\pi)$ equals
\begin{align}
n  =~ & \frac{1}{\pi}\Bigg[\int_0^{2\pi}\frac{f'(x)^2-f(x)f''(x)}{f(x)^2+f'(x)^2}\,\mathrm dx+\lim_{x\nearrow 2\pi}\arctan\left(\frac{f'(x)}{f(x)}\right)\nonumber \\ & -\lim_{x \nearrow 0}\arctan\left(\frac{f'(x)}{f(x)}\right)\Bigg]\nonumber=\\
 =~ & \frac{1}{\pi}\int_0^{2\pi}\frac{f'(x)^2-f(x)f''(x)}{f(x)^2+f'(x)^2}\,\mathrm dx,\label{eq-int}
\end{align}
since the integral-free terms cancel out in this case. In this way we obtain Lemma \ref{lemmaidee} as a corollary of Theorem \ref{main}. Observe that a $2\pi$-periodic $C^2$ function with an odd number of zeros on $[0,2\pi)$ gives rise to a curve $x\mapsto(f'(x),f(x))$ having a half-integer valued winding number. This idea, further developed, leads to a generalized version of the Residue Theorem (see~\cite{residuenpaper}).
\end{enumerate}

%\begin{remark} Observe that Theorem \ref{main} applies also to functions of class $C^0([a,b])$ having a continuous second derivative on $[a,b]$ at the exception of finitely many points where it may not even be defined.\end{remark}

Observe, that for a $C^2$ function $f$ with only zeros of multiplicity one, the integrand in~(\ref{eq-int}) is continuous provided $h$ is continuous. This remains true for zeros of higher multiplicity in the following way:
\begin{prop}\label{prop-cont}
Let $h:\mathbb R\to\mathbb R$ be continuous and $h(x)\sim \frac C{x^2}$ for $|x|\to\infty$.
Then, the integrand in Theorem~\ref{main}
$$
\operatorname{I}:=h\left(\frac{f'}{f}\right)\frac{f'^2-ff''}{f^2}
$$
is continuous if $f\in C^n([a,b])$, $n\geqslant2$, has only zeros of multiplicity $\leqslant n$.
\end{prop}
\begin{proof}
It suffices to show that $\operatorname{I}$ is continuous in $0$ if $x=0$ is a zero of $f$
of multiplicity $n$. Then, by Taylor expansion, we have
\begin{eqnarray*}
f(x)&=& \left(\frac{f^{(n)}(0)}{n!}+r_0(x)\right)x^n\\
f'(x)&=& \left(\frac{f^{(n)}(0)}{(n-1)!}+r_1(x)\right)x^{n-1}\\
f''(x)&=& \left(\frac{f^{(n)}(0)}{(n-2)!}+r_2(x)\right)x^{n-2}
\end{eqnarray*}
where $r_i$ are continuous functions with $\lim_{x\to 0}r_i(x)=0$. Using
these expressions in $\operatorname{I}$, we get
$$
\operatorname{I}(x)=h\left(\frac{s_1(x)}x\right)\frac{s_2(x)}{x^2}
$$
for continuous functions $s_i$ with $\lim_{x\to 0}s_i(x)=n$. Thus
$$
\operatorname{I}(x)\sim \frac{Cx^2}{s_1^2(x)}\frac{s_2(x)}{x^2}\to \frac Cn
$$
for $x\to0$.
\end{proof}
If we only assume that $h(x)=O(1/x^{2})$ for $|x|\to\infty$ in the previous proposition, the proof
shows that then $\operatorname{I}$ is at least bounded.

As a corollary of Proposition \ref{prop-cont} we obtain that if $h$ is continuous and $h(x)\sim \frac{C}{x^2}$, then $\operatorname{I}$ is in $C^0$ provided $f$ is analytic. 
Nontheless, the function $f$ may behave in the neighborhood
of a zero in such a pathological way, that $\operatorname{I}$ becomes unbounded (see Example~\ref{ex-2}.\ref{remarkmultiplicities}). 
This is why, in general, the integrals in Theorem~\ref{main} have to be interpreted as improper integrals.
This means that the concrete computation 
requires the zeros of $f$ to be known a priori in order to evaluate the improper integrals. 
It is therefore of practical importance to formulate conditions (see Propositions~\ref{prop-crit} and~\ref{prop-crit2}) with additional
assumptions which guarantee that $\operatorname{I}$ is in $L^1$: To this end we will slightly sharpen the admissibility 
condition for a function and impose some conditions on the behaviour of the zeros of $f''$ in 
neighborhoods of the zeros of $f$. Furthermore we will require $h$ to have at least quadratic 
decay at infinity.

The proof of Proposition~\ref{prop-cont} for the case $C=1$
indicates, how we can generalize the notion of multiplicity of zeros in a natural manner:
\begin{definition}\label{def-mult}
The multiplicity $\mu_f(x_0)$ of a zero $x_0$ of $f\in\mathcal A^0$ is defined to be
$$
\mu_f(x_0)=\lim_{x\to x_0}\frac{f'(x)^2}{f'(x)^2-f(x)f''(x)}.
$$
\end{definition}
Since the zeros of functions in $\mathcal A^0$ are admissible, it follows that $\mu_f(x_0)\geqslant 0$ whenever it exists, however, it can take values in $[0,\infty]$ (see Example \ref{ex-2}.\ref{remarkmultiplicities} and \ref{ex-2}.\ref{infinitemultiplicity} below). This definition of the multiplicity of a zero will be useful for a variant of Theorem \ref{main} that takes the multiplicities of the zeros into account.
\begin{example}\label{ex-2}
\begin{enumerate}
\item A function $f\in C^n$, $n\geqslant 2$ with $0=f(x_0)=f'(x_0)=\ldots=f^{(n-1)}(x_0)\neq f^{(n)}(x_0)$
has a zero of multiplicity $n$ in $x_0$: the Definition~\ref{def-mult} is compatible
with the usual notion of multiplicity.
\item The function $f(x) = |x|^r,~r>0$ has a zero of multiplicity $r$ in $x=0$.
\item \label{remarkmultiplicities} The function
$$
f(x) = \begin{cases}\displaystyle\frac{1}{\ln|x|},&x\ne0\\
\hfill 0,& x=0\end{cases}
$$
has a zero of multiplicity $0$ in $x=0$.
\item\label{infinitemultiplicity} The function $f_2$ in Example \ref{example-zero}.\ref{expzero} has a zero in $x=0$ with $\mu_{f_2}(0)=\infty$.
\end{enumerate}
\end{example}

%Each equivalence class of functions in $W^{2,1}(a,b)$ has a pointwise
%representative $f\in C^1$ with an absolutely continuous (and therefore uniformly continuous) derivative $f'$. In particular, $f'$ extends in a unique way to a function in $C^0([a,b])$ and we obtain $W^{2,1}(a,b)\subset C^1([a,b])$.
%In what follows we will tacitly identify $f\in W^{2,1}(a,b)$ with its representative in $C^1([a,b])$.

\begin{prop}\label{prop-crit}
Let $h:\R\to\R$ be a piecewise continuous function such that $h(x)=O(1/x^{2})$ for $|x|\to\infty$ and let $f\in\mathcal A^0([a,b])\cap W^{2,1}(a,b)$ have only zeros of positive multiplicity in the sense of Definition \ref{def-mult}.
Furthermore we assume that for each zero $x_0$ we have a neighborhood $U$ such that either $f''(x)\equiv 0$ on $U\setminus\{x_0\}$ or
$$
\sum_{k=1}^\infty |z_k-x_0|<\infty,
$$
where $z_1,z_2,\ldots $ denote the countably many zeros of $f''$ in $U$.
Then
$$
\mathrm I:=h\left(\frac{f'}{f}\right)\frac{f'^2-ff''}{f^2}\in L^1(a,b).
$$
\end{prop}
\begin{proof}
Choose neighborhoods $U_1,\ldots, U_n$ of the $n$ zeros of $f$, which do not (with the possible exception of the respective zero itself) contain singular points of $f''$ or zeros of $f'$ and let
$$
\mathrm U = \bigcup_{i=1}^n U_i.
$$
Since $|f|\geqslant \eta$ for some $\eta>0$ on the complement $\mathrm U^c$ and $W^{2,1}(a,b)\hookrightarrow C^1([a,b])$
we can estimate$$
\int_{\mathrm U^c}|\mathrm I(x)|\,\mathrm dx = \eta^{-2}\|h\|_{L^\infty(\R)}\Big(\left\|f'^2\right\|_{C^0([a,b])}|b-a|+\|f\|_{C^0([a,b])}\|f''\|_{L^1(a,b)}\Big)<\infty.
$$
Consider now wlog the neighborhood $U_i$ of the zero $x_i=0$ and assume $U_i=(-\varepsilon,\varepsilon)$ for some $\varepsilon>0$. We need to show that $\mathrm I\vert_{(-\varepsilon,\varepsilon)}\in L^1$. Since $h(x)=O(1/x^{2})$ for $|x|\to\infty$, there exists a constant $C>0$ such that
\begin{equation}\label{decayestimate}
|\mathrm I(x)|\leqslant C\left(1+\left|\frac{f(x)f''(x)}{f'(x)^2}\right|\right).
\end{equation}
Note that $f f''/f'^2\in L^1(-\varepsilon,\varepsilon)$ if and only if $\mathrm N\in \mathrm{BV}(-\varepsilon,\varepsilon)$, where $\mathrm N(x) = x-f(x)/f'(x)$ denotes the \emph{Newton-Operator} of $f$ and $\mathrm{BV}(-\varepsilon,\varepsilon)$ denotes the space of functions $g:(-\varepsilon,\varepsilon)\to\R$ of bounded variation. It follows from the admissibility of the zero that $\mathrm N:(-\varepsilon,\varepsilon)\setminus\{0\}\to\R$ can be continuously extended to $\mathrm N(0)=0$ and it holds that
$$
\mathrm N'(x) = \frac{f(x)f''(x)}{f'(x)^2},
$$
for $x\ne 0$. Let $\mu>0$ denote the multiplicity of the zero according to Definition \ref{def-mult}. It holds that
$$
\lim_{x\to0}\mathrm N'(x)=\begin{cases}\displaystyle\hfill\frac{\mu-1}{\mu},& \mu < \infty \\ \hfill 1,&\mu=\infty.\end{cases}
$$
%if $\mu<\infty$ and $\lim_{x\to0}\mathrm N'(x)=1$ if $\mu=\infty$.
According to the mean value theorem we have $\mathrm N(x)/x = \mathrm N'(\xi)$ for some $\xi$ between $0$ and $x$ and deduce that $\mathrm N\in C^1(-\varepsilon,\varepsilon)$. The Taylor expansion of $\mathrm N$ around $x=0$ is given by
$$
\mathrm N(x) =\begin{cases}\displaystyle\frac{\mu -1}{\mu} x + o(x),&\mu<\infty \\ \hfill x+o(x),&\mu=\infty.\end{cases}
$$
In any case there exists a constant $K>0$ such that
\begin{equation}\label{newton-estimate}
|\mathrm N(x)| \leqslant K|x|,\quad |x|<\varepsilon.
\end{equation}

We will now show that $\mathrm N\in \mathrm{BV}([0,\varepsilon))$, the argument on $(-\varepsilon,0]$ being similar. We start by noticing that $\mathrm N$ is absolutely continuous on $[\delta,\varepsilon)$ for every $0<\delta<\varepsilon$ since $x$, $f(x)$ and $f'(x)$ are absolutely continuous and $f'(x)\ne 0$ on $[\delta,\varepsilon)$. In particular, $\mathrm N\in \mathrm{BV}([\delta,\varepsilon))$ for every $0<\delta<\varepsilon$.

We will now distinguish two cases: If $f''\equiv 0$ on $(0,\varepsilon)$, then $\mathrm N \equiv 0$ and we are done. In the remaining case we first consider the case when the set of zeros of $f''$ in $(0,\varepsilon)$ is empty: Then $\mathrm N$ is monotone on $[0,\varepsilon)$ and hence $\mathrm N\in\mathrm{BV}([0,\varepsilon))$. Otherwise the zeros of $f''$ in $[0,\varepsilon)$ are given by $z_1>z_2>\ldots$ and we may set $\delta:=z_1$. According to \eqref{newton-estimate} and since the zeros of $f''$ are precisely the zeros of $\mathrm N'$ we can estimate the total variation of $\mathrm N$ on $(z_{k+1},z_k)$ by
$$
\int_{z_{k+1}}^{z_{k}} |\mathrm N'(x)|\,\mathrm dx\leqslant 2Kz_k.
$$
The total variation of $\mathrm N$ on $[0,\varepsilon)$ is bounded by
$$
\sum_{k = 1}^\infty \int_{z_{k+1}}^{z_{k}} |\mathrm N'(x)|\,\mathrm dx + \int_{\delta}^\varepsilon |\mathrm N'(x)|\,\mathrm dx\leqslant 2K\sum_{k=1}^\infty z_k + \int_{\delta}^\varepsilon |\mathrm N'(x)|\,\mathrm dx,
$$
where the series converges by assumption and the integral is finite since $\mathrm N\in\mathrm{BV}([\delta,\varepsilon))$.
We conclude that $\mathrm N\in\mathrm{BV}([0,\varepsilon))$, which finishes the proof.
\end{proof}

\emph{Remark.}
The key estimate \eqref{newton-estimate} in the proof above follows from the admissibility and the positive multiplicity of the zeros. We will however formulate a variant of Proposition \ref{prop-crit} below (Proposition \ref{prop-crit2}), which covers admissible functions that have zeros of ill-defined multiplicity for which \eqref{newton-estimate} still holds true: Take e.g.\ the $C^1$ function $f:x\mapsto x^3 \left(\sin(1/x) + 2\right) + x$ which has an admissible zero in $x=0$, but for which $\mu_f(0)$ does not exist, however, \eqref{newton-estimate} holds true since $f(x)/(xf'(x))$ is bounded near $0$ -- in fact
$$
\lim_{x\to 0}\frac{f(x)}{xf'(x)}=1.
$$
Example~\ref{ex-2}.\ref{remarkmultiplicities} shows an admissible function for which \eqref{newton-estimate} does not hold true. In the mentioned example, the first derivative is unbounded. But even functions with higher regularity may behave in such a pathological way near an admissible zero, that \eqref{newton-estimate} does not hold true, as the following example shows:

\begin{example}\label{ex-counter}
Let $$
k(x) = \begin{cases} \hfill x^3+\left(\sqrt{|x|^7}-x^3\right)\cos\left(\pi\log_2|x|\right), & \text{if $x\ne 0$}\\
\hfill 0,& \text{if $x=0$.}\end{cases}$$
Then $f(x) = \int_0^x k(t)\,\mathrm dt$ is of class $C^3$ and has an admissible zero in $x=0$ but $f(x)/(xf'(x))$ is unbounded near 0.
\end{example}

%HIER NOCHMALS DEN BEWEIS STUDIEREN!!!

\begin{prop}\label{prop-crit2}
Let $h:\R\to\R$ be a piecewise continuous function such that $h(x)=O(1/x^{2})$ for $|x|\to\infty$ and let $f\in\mathcal A^0([a,b])\cap W^{2,1}(a,b)$ be such that that for every zero $x_0$ of $f$ there exists a relatively open neighborhood $U\subset[a,b]$ such that
\begin{equation}\label{condition}
0<\left|\frac{f(x)}{(x-x_0) f'(x)}\right|<\widetilde K
\end{equation}
on $U\setminus\{x_0\}$ and
such that either $f''\equiv 0$ on $U\setminus\{x_0\}$, or
$$
\sum_{k=1}^\infty |z_k-x_0|<\infty,
$$
where $z_1,z_2,\ldots $ denote the countably many zeros of $f''$ in $U\setminus\{x_0\}$.
Then
$$
\mathrm I:=h\left(\frac{f'}{f}\right)\frac{f'^2-ff''}{f^2}\in L^1(a,b).
$$
\end{prop}
\begin{proof}
Choose neighborhoods $U_1,\ldots, U_n$ of the $n$ zeros of $f$, which do not (with the possible exception of the respective zero itself) contain singular points of $f''$ or zeros of $f'$ such that \eqref{condition} holds on each punctured neighborhood. As in the proof of Proposition \ref{prop-crit} we obtain $\|\mathrm I\|_{L^1(\mathrm U^c)}<\infty$, where $\mathrm U = U_1\cup \ldots \cup U_n$ and the estimate \eqref{decayestimate}. Let wlog $0$ be a zero of $f$ and let $(-\varepsilon,\varepsilon)$ be its respective neighborhood for some $\varepsilon>0$. As in the proof of Proposition \ref{prop-crit}, we are done if we show that $\mathrm N\in\mathrm{BV}([0,\varepsilon))$. The condition $0<|f(x)/(x f'(x))|<\widetilde K$ on $(-\varepsilon,\varepsilon)\setminus\{0\}$ implies that 
\begin{equation}\label{condition2}
0<\left|\frac{f(x)}{f'(x)}\right|< \widetilde K|x|,
\end{equation}
from which we conclude that $\mathrm N$ extends continuously to $[0,\varepsilon)$ (where $\mathrm N(0)=0$) and
\begin{equation}\label{boundonn}
|\mathrm N(x)| \leqslant (\widetilde K+1)\,x,\quad x\in[0,\varepsilon).
\end{equation}
This is just estimate \eqref{newton-estimate} with $K=\widetilde K+1$. The rest of the proof is exactly the same as the one of Proposition \ref{prop-crit}.
\end{proof}

\section{Counting Zeros with Multiplicities}
Let again $h:\R\to\R$ be a piecewise continuous function such that $\int_{-\infty}^{\infty}h(x)\,\mathrm dx = 1$ and define $\mathrm H$ as before in~(\ref{eq-H}). Moreover, let

$$\begin{aligned}
\mathrm I_g(x) &=h\left(\frac{f'(x)}{f(x)}\right)g(x)\frac{f'(x)^2-f(x)f''(x)}{f(x)^2}-\mathrm H\left(\frac{f'(x)}{f(x)}\right) g'(x),\\
g_1(x) &= \frac{f'(x)^2}{f'(x)^2-f(x)f''(x)+cf(x)^2},\\
g_2(x)& =\exp\left(\frac{f'(x)^2-f(x)f''(x)}{f'(x)^2+f(x)^2}\right),
\end{aligned}$$
where $c\in\R$.
Note that if $x_0$ is a zero of multiplicity $\mu_f(x_0)$, then $g_1(x) \to \mu_f(x_0)$ as $x\to x_0$ for every value $c$ in the definition of $g_1$ and if $\mu_f(x_0)>0$, then $g_2(x) \to \exp\big(\frac{1}{\mu_f(x_0)}\big)$ as $x\to x_0$.

%This definition of a multiplicity of a zero is justified by the following lemma:
%\begin{lemma}
%Let $f:[a,b]\to\R$ be analytic. If $f$ has a zero of multiplicity $k$ in $x_0\in[a,b]$, then $
%\lim\limits_{x\to x_0}g_1(x)=k$ for every $c\in\R$ and $
%\lim\limits_{x\to x_0}g_2(x)=\exp\left({\frac{1}{k}}\right).$
%\end{lemma}
%\begin{proof}
%In a neighborhood of $x_0$ we may write $f(x) = (x-x_0)^kj(x)$, where $j$ is an analytic function with $j(x_0)\ne 0$. The first limit then follows from
%$$
%g_1(x) = \frac{(kj(x)+(x-x_0)j'(x))^2}{(k+c(x-x_0)^2)j(x)^2+(x-x_0)^2(j'(x)^2-j(x)j''(x))}
%$$
%and the second one from
%$$
%g_2(x)=\frac{kj(x)^2+(x-x_0)^2(j'(x)^2-j(x)j''(x))}{(k^2+(x-x_0)^2)j(x)^2+2k(x-x_0)(j(x)j'(x)+(x-x_0)j'(x)^2)}.
%$$
%\end{proof}

\begin{lemma}
Let all the zeros of $f\in\mathcal A^0([a,b])\cap C^2([a,b])$ have well-defined multiplicities. Then there exists $c\in \R$ such that $g_1$ has no poles.
\end{lemma}
\begin{proof}
If $x_0$ is a zero of $f$, we have that $g_1(x)\to\mu_f(x_0)$ as $x\to x_0$. In other words $g_1$ extends continuously to the zeros of $f$. Hence there are open neighborhoods of the zeros of $f$, where $g_1$ has no poles. On the complement of these neighborhoods, there exists a number $\delta>0$ such that $|f(x)|\geqslant\delta$. Hence $f'(x)^2+cf(x)^2\geqslant f'(x)^2+c\delta^2$. If we choose $c>\delta^{-2}\|ff''\|_{C^0([a,b])}$, then $g_1$ has no poles. In particular, if $f$ is analytic, this choice of $c$ ensures that $g_1$ is analytic as well.\end{proof}

We have the following theorem for analytic functions $f:[a,b]\to\R$:

\begin{theorem}\label{mainanalytic}
Let $f:[a,b]\to\R$ be an analytic function and choose $c$ in the definition of $g_1$ such that $g_1$ is analytic. If $h(x) = O(1/x^2)$ for $|x|\to\infty$, then $\mathrm I_{g_1}, \mathrm I_{g_2}\in L^\infty(a,b)$ and if $f$ has $n_\ell$ zeros of multiplicity $\ell$ in $[a,b]$ and $\mathring n_\ell$ zeros of mutliplicity $\ell$ in $(a,b)$, then
$$\begin{aligned}
\int_a^{b}\mathrm I_{g_1}(x)\,\mathrm dx + \lim_{x\searrow b}\left[\mathrm H\left(\tfrac{f'(x)}{f(x)}\right)g_1(x)\right]-\lim_{x\nearrow a}\left[\mathrm H\left(\tfrac{f'(x)}{f(x)}\right)g_1(x)\right] & =\sum_{\ell=1}^\infty n_\ell \ell,\\
\int_a^{b}\mathrm I_{g_1}(x)\,\mathrm dx + \lim_{x\nearrow b}\left[\mathrm H\left(\tfrac{f'(x)}{f(x)}\right)g_1(x)\right]-\lim_{x\searrow a}\left[\mathrm H\left(\tfrac{f'(x)}{f(x)}\right)g_1(x)\right] & =\sum_{\ell=1}^\infty  \mathring n_\ell \ell,\\
\int_a^{b}\mathrm I_{g_2}(x)\,\mathrm dx + \lim_{x\searrow b}\left[\mathrm H\left(\tfrac{f'(x)}{f(x)}\right)g_2(x)\right]-\lim_{x\nearrow a}\left[\mathrm H\left(\tfrac{f'(x)}{f(x)}\right)g_2(x)\right]&=\sum_{\ell=1}^\infty n_\ell\exp\left(\tfrac{1}{\ell}\right),\\
\int_a^{b}\mathrm I_{g_2}(x)\,\mathrm dx + \lim_{x\nearrow b}\left[\mathrm H\left(\tfrac{f'(x)}{f(x)}\right)g_2(x)\right]-\lim_{x\searrow a}\left[\mathrm H\left(\tfrac{f'(x)}{f(x)}\right)g_2(x)\right]&=\sum_{\ell=1}^\infty \mathring n_\ell\exp\left(\tfrac{1}{\ell}\right).
\end{aligned}$$
\end{theorem}

\begin{proof}
We first prove the $L^\infty$-bounds: It suffices to show that $\mathrm I_{g_1}$ and $\mathrm I_{g_2}$ are bounded near the zeros of $f$. Let $x_0$ be a zero of multiplicity $k$ and write (locally) $f(x) = (x-x_0)^kj(x)$, where $j$ is analytic and $j(x_0)\ne0$.
Since
$$
\lim_{x\to x_0}g_1'(x) = \frac{2j'(x_0)}{j(x_0)}
$$
we find the limits
$$\begin{aligned}
\lim_{x\searrow x_0}\mathrm H\left(\frac{f'(x)}{f(x)}\right) g_1'(x) & =  \frac{2j'(x_0)}{j(x_0)}\\
\lim_{x\nearrow x_0}\mathrm H\left(\frac{f'(x)}{f(x)}\right) g_1'(x) & =  0.
\end{aligned}$$
If
$$
h\left(\frac{f'(x)}{f(x)}\right)g(x)\frac{f'(x)^2-f(x)f''(x)}{f(x)^2}
$$
is bounded near $x_0$, the claim follows. Since $\left|h\left({f'(x)}/{f(x)}\right)\right|\leqslant C\cdot{f(x)^2}/{f'(x)^2}$ and
$$
\lim_{x\to x_0}C |g_1(x)| \frac{f'(x)^2+|f(x)f''(x)|}{f'(x)^2} = C(2k-1) ,
$$
we obtain $\mathrm I_{g_1}\in L^\infty(a,b)$. For $\mathrm I_{g_2}$, observe that
$$
\lim_{x\to x_0}g_2'(x) = -\frac{2\exp\left(\tfrac{1}{k}\right)j'(x_0)}{k^2j(x_0)}
$$
and therefore
$$\begin{aligned}
\lim_{x\searrow x_0}\mathrm H\left(\frac{f'(x)}{f(x)}\right) g_2'(x) & = -\frac{2\exp\left(\tfrac{1}{k}\right)j'(x_0)}{k^2j(x_0)}\\
\lim_{x\nearrow x_0}\mathrm H\left(\frac{f'(x)}{f(x)}\right) g_2'(x) & =  0.
\end{aligned}$$
Proceeding as for $g_1$ we find
$$
\lim_{x\to x_0}C |g_2(x)| \frac{f'(x)^2+|f(x)f''(x)|}{f'(x)^2} = C\exp\left(\tfrac{1}{k}\right)\frac{2k-1}{k}
$$
and hence $\mathrm I_{g_2}\in L^\infty(a,b)$. The computation of the integrals is done as in the proof of Theorem \ref{main}.
\end{proof}

\emph{Remark.}
 If $f\in\mathcal A^1([a,b])\cap C^2([a,b])$ only has zeros of well-defined multiplicities and if the set of zeros of $f$ in $(a,b)$ is given by $\mathring{ N}$ and the set of zeros of $f$ in $[a,b]$ by $ N$, then
$$\begin{aligned}
\int_a^{b}\mathrm I_{g_1}(x)\,\mathrm dx + \lim_{x\searrow b}\left[\mathrm H\left(\tfrac{f'(x)}{f(x)}\right)g_1(x)\right]-\lim_{x\nearrow a}\left[\mathrm H\left(\tfrac{f'(x)}{f(x)}\right)g_1(x)\right] & =\sum_{x\in { N}} \mu_f(x),\\
\int_a^{b}\mathrm I_{g_1}(x)\,\mathrm dx + \lim_{x\nearrow b}\left[\mathrm H\left(\tfrac{f'(x)}{f(x)}\right)g_1(x)\right]-\lim_{x\searrow a}\left[\mathrm H\left(\tfrac{f'(x)}{f(x)}\right)g_1(x)\right] & =\sum_{x\in \mathring{ N}} \mu_f(x).
\end{aligned}$$
%\textcolor{red}{If in addition $\mu(x)>0$ for every $x\in\mathrm N$, then it holds that
%$$\begin{aligned}
%\int_a^{b}\mathrm I_{g_2}(x)\,\mathrm dx + \lim_{x\searrow b}\left[\mathrm H\left(\tfrac{f'(x)}{f(x)}\right)g_2(x)\right]-\lim_{x\nearrow a}\left[\mathrm H\left(\tfrac{f'(x)}{f(x)}\right)g_2(x)\right]&=\sum_{x\in {\mathrm N}} \exp\left(\tfrac{1}{\mu(x)}\right),\\
%\int_a^{b}\mathrm I_{g_2}(x)\,\mathrm dx + \lim_{x\nearrow b}\left[\mathrm H\left(\tfrac{f'(x)}{f(x)}\right)g_2(x)\right]-\lim_{x\searrow a}\left[\mathrm H\left(\tfrac{f'(x)}{f(x)}\right)g_2(x)\right]&=\sum_{x\in \mathring{\mathrm N}} \exp\left(\tfrac{1}{\mu(x)}\right).\end{aligned}$$}

\begin{lemma}
Let $\mathcal N$ be the set of sequences with natural entries of which only finitely many are non-zero. Then the map $\mathcal F:\mathcal N\to\R$ defined by $
\mathcal F(k_1,\ldots) = \sum_{\ell=1}^\infty k_\ell \exp\left(\tfrac{1}{\ell}\right)$ is injective.
\end{lemma}
\begin{proof}
The difference $\mathcal F(k_1,\ldots )-\mathcal F(k'_1,\ldots )$ is equal to the finite sum $$
\sum_{\ell=1}^\infty(k_\ell-k'_\ell) \exp\left(\tfrac{1}{\ell}\right).$$ If this sum vanishes, $k_\ell=k'_\ell$ for all $\ell$ by the von~Lindemann-Weierstrass theorem (see~\cite[\S 3]{weierstrass}).
\end{proof}
\begin{coro}
Let $f:[a,b]\to \R$ be analytic. If $f$ has $n_\ell$ zeros of multiplicity $\ell$ in $[a,b]$ and $\mathring n_\ell$ zeros of mutliplicity $\ell$ in $(a,b)$, then
$$\begin{aligned}
(n_1,\ldots)& =\mathcal F^{-1}\left(\int_a^{b}\mathrm I_{g_2}(x)\,\mathrm dx + \lim_{x\searrow b}\left[\mathrm H\left(\tfrac{f'(x)}{f(x)}\right)g_2(x)\right]-\lim_{x\nearrow a}\left[\mathrm H\left(\tfrac{f'(x)}{f(x)}\right)g_2(x)\right]\right)\\
(\mathring n_1,\ldots)& = \mathcal F^{-1}\left(\int_a^{b}\mathrm I_{g_2}(x)\,\mathrm dx + \lim_{x\nearrow b}\left[\mathrm H\left(\tfrac{f'(x)}{f(x)}\right)g_2(x)\right]-\lim_{x\searrow a}\left[\mathrm H\left(\tfrac{f'(x)}{f(x)}\right)g_2(x)\right]\right).\end{aligned}
$$
\end{coro}

\begin{example}
Let $f(x)=\cos(2 x) + x^2 \sin(2 x) -\frac{1}{2} \sqrt{\mathrm e^{x}}+\frac{x-2}{4}$. Using theorem \ref{main} and \ref{mainanalytic} on $[0,2\pi]$ we obtain
$$
\sum_{\ell = 1}^\infty n_\ell = 3,\quad
\sum_{\ell = 1}^\infty \mathring n_\ell  = 2,\quad
\sum_{\ell = 1}^\infty n_\ell\ell  = 4,\quad
\sum_{\ell = 1}^\infty \mathring n_\ell\ell  = 2.
$$
and we conclude that $f$ has two zeros in $(0,2\pi)$ and a double zero on the boundary of $[0,2\pi]$.
\end{example}
\begin{example}
Let $f(x) = x^7-2x^6+x^5-x^3+2x^2-x$ have $n_\ell$ zeros of multiplicity $\ell$ on $\R$. By Theorem \ref{main} and \ref{mainanalytic} on $\R$ (observe that the boundary terms of the integrals cancel out in this case) we find that
$$\sum_{\ell=1}^\infty n_\ell  = 3 \text{ and }\sum_{\ell=1}^\infty  n_\ell\ell  = 5.$$
Hence $(n_1,\ldots)$ either equals $(1,2,0,\ldots)$ or $(2,0,1,\ldots)$. In particular $n_\ell = 0$, for $\ell\geqslant 4$. Using again Theorem \ref{mainanalytic} we get
$$
\sum_{i=1}^3 n_i \exp\left(\tfrac{1}{i}\right)\approx 6.8322.
$$
Since $1\cdot \mathrm e + 2 \cdot\sqrt{\mathrm e}\approx 6.0157$ and $2\cdot \mathrm e + 1\cdot \sqrt[3]{\mathrm e} \approx 6.8322$ we conclude that $f$ has two simple zeros and one of multiplicity 3.
\end{example}
\section{Numerical Aspects}
The number of zeros of a function $f$ in a given interval $[a,b]$ is of course an integer.
Therefore is suffices to compute the integral in Theorem~\ref{main} with
an error $\varepsilon<\frac12$. In particular, for the trapezoidal rule
$$T_N(\mathrm I):=\frac{b-a}N\Bigl(\frac{\mathrm I(a)+\mathrm I(b)}2+\sum_{k=1}^{N-1}\mathrm I\bigl(a+k\tfrac{b-a}N\bigr)\Bigr)$$
with $N+1$ equidistant grid points, the error $\varepsilon(N)$ is estimated by
$$
\varepsilon(N)=\Bigl|\int_a^b \mathrm I(x)\,\mathrm dx -T_N(\mathrm I)\Bigr|\leqslant\frac{(b-a)^3}{12N^2}\Vert \mathrm I''\Vert_{L^\infty}
$$
(see, e.g.,~\cite{schwarz} or~\cite{gautschi}).
Thus we have
\begin{theorem}\label{thm-numerics}
Let $f$ satisfy the assumptions in Theorem~\ref{main}. If 
$$
N>\sqrt{\frac{(b-a)^3}{6}\Vert \mathrm I''\Vert_{L^\infty}}\,,
$$
then one can replace the integral in Theorem~\ref{main} by the finite sum $T_N(\mathrm I)$
and round the result to the closest integer to get the values $n(f)$ and $\mathring n(f)$,
respectively.
\end{theorem}
This theorem is quite remarkable: It allows to compute the number of zeros 
of a function $f$ on $[a,b]$ by evaluating finitely many values of $f,f'$ and  $f''$.

\begin{example}\label{numericalex}
Let $f:\mathbb R\to\mathbb R, ~x\mapsto J_0(x)$, be the zeroth Bessel function of the first kind. 
If $h$ is the Cauchy density, one can verify that $\|\mathrm I''\|_{L^\infty}<\frac{1}{\pi}$. 
We want to compute the number of zeros of $J_0$ on $[0,2\pi]$ by Theorem \ref{thm-numerics}.
It suffices to employ the trapezoidal rule with only
$$N=\left\lceil \frac{2\pi}{\sqrt3}\right\rceil = 4$$
equidistant intervals. We find
$$\begin{aligned}
T_{4}(\mathrm I)= \frac{\pi}{2}\left(\frac{\mathrm I(0) + \mathrm I(2\pi)}{2}+\sum_{k=1}^{3}\mathrm I\big(k\frac{\pi}{2}\big) \right)\approx 1.76479
\end{aligned}$$
and thus
$$
T_{4}(\mathrm I) - \frac1\pi\arctan\left( \frac{J_1(2\pi)}{J_0(2\pi)}\right)  \approx 1.76479 + 0.24419= 2.00898$$
and hence, $J_0$ has two zeros on $[0,2\pi]$.

If we compute the number of zeros of $J_0$ on $[0,100\pi]$, we have to choose
$$
N=\left\lceil \frac{500\sqrt{6}}{3}\cdot\pi\right\rceil = 1283.$$
(Actually, a finer analysis shows that a much smaller $N$ suffices). In this case, we get
$$\begin{aligned}
T_{1283}(\mathrm I)= \frac{100\pi}{1283}\left(\frac{\mathrm I(0) + \mathrm I(100\pi)}{2}+\sum_{k=1}^{1282}\mathrm I\big(k\frac{100\pi}{1283}\big) \right)\approx 99.75013
\end{aligned}$$
and
$$
T_{1283}(\mathrm I) - \frac1\pi\arctan\left( \frac{J_1(100\pi)}{J_0(100\pi)}\right) \approx 99.75013 + 0.24987 = 100,$$
hence we conclude that $J_0$ has $n=100$ zeros on $[0,100\pi]$, in accordance with the well known distribution of zeros of $J_0$. Surprisingly, the routine {\textsf{CountRoots}}
of {\textsf{Mathematica}}\texttrademark{} is giving up on this simple problem after giving it some thought.
\end{example}

From a practical point of view, it is desirable to keep $\Vert\mathrm I''\Vert_{L^\infty}$ (and hence $N$)
as small as possible. This can be achieved in several ways:
First of all, we have the freedom to choose the function $h$. Here is a small table
of possible choices of $h$ and the resulting function $\mathrm H$ in Theorem~\ref{main} (in
each case, the integrand $\mathrm I$ turns out rather nicely):
\begin{center}
\def\arraystretch{2}
$\begin{tabular}{|L|L|}
\hline

h(x) & \mathrm H(x)\\[0.5em]\hline

 \frac1{\pi (1 + x^2)}& \frac{\arctan x}\pi\\[0.5em]

\hline

\frac{1}{2 \left(x^2+1\right)^{3/2}}& \frac{x}{2 \sqrt{x^2+1}}\\[0.5em]

\hline

\frac{\exp \left(-x^2\right)}{\sqrt{\pi }}&  \frac12\operatorname{erf}(x)\\[0.5em]\hline

\frac1{4x^2}-\frac{1}{4x^2\sqrt{4 x^2+1}}&  \frac{\sqrt{4 x^2+1}-1}{4 x}\\[0.5em]
\hline

\operatorname{sech}(2 x)^2& \frac12\tanh(2x) \\[0.5em]
\hline

\frac{\mathrm e^x}{(1+\mathrm e^x)^2}& -\frac{1}{1+\mathrm e^x}\\[0.5em]
\hline

\operatorname{UnitBox}(x)& \begin{cases}
0 & \text{if $2 x<-1$} \\
 x+\frac{1}{2} &\text{if  $-\frac{1}{2}<x\leqslant \frac{1}{2}$} \\
 1 & \text{if $2 x>1$}
\end{cases}\\[0.5em]
\hline
 \operatorname{UnitTriangle}(x)& \begin{cases} 
0&\text{if $x\leqslant-1$}\\
\frac12(1+x)^2&\text{if $-1<x\leqslant 0$}\\
-\frac12(1-x)^2+1&\text{if $0<x\leqslant 1$}\\
1&\text{if $1<x$}
\end{cases}\\[0.5em]
\hline
\end{tabular}$
\end{center}
Moreover, with smooth functions $\gamma$ and $\kappa$ that satisfy $\operatorname{sign}\gamma(x)=\operatorname{sign}\kappa(x)=\operatorname{sign} x$ for all $x\neq 0$ and $\gamma(x)\sim C_1 |x|^\alpha\operatorname{sgn}x$ and $\kappa(x)\sim C_2 |x|^\beta\operatorname{sgn}x$ as $x\to 0$, where $0<\alpha\leqslant\beta$, one can modify the integrand $\mathrm I$ as follows and
the proof of Theorem~\ref{main} still goes through:
$$
\mathrm I(x)=h\left(\frac{\gamma (f'(x))}{\kappa(f(x))}\right)\left(\frac{\gamma (f'(x))f'(x)\kappa'(f(x))-\gamma '(f'(x))f''(x)\kappa(f(x))}{\kappa(f(x))^2}\right).
$$
In this case the boundary terms in $a$ and $b$ have to be taken with the function
$$
\mathrm H\left(\frac{\gamma (f'(x))}{\kappa(f(x))}\right).
$$
%\bibliography{zeros}
%\bibliographystyle{plain}
\section*{Acknowledgement}
We would like to thank the referees for their valuable remarks 
which greatly helped to improve this article.

\end{document}